\documentclass{article}
\usepackage{amsmath}
\usepackage{amssymb}
\usepackage{amsfonts}
\usepackage{amsthm}
\usepackage[margin=4cm]{geometry}
\usepackage{hyperref}
\usepackage{fixmath}
\usepackage{aliascnt}
\usepackage{comment}
\usepackage{pictexwd,dcpic}
\usepackage{natbib}
\usepackage{color}
\usepackage{graphicx}
\theoremstyle{plain}
\newtheorem{theorem}{Theorem}[section] 
\theoremstyle{definition}
\theoremstyle{property}
\theoremstyle{remark}
\newaliascnt{proposition}{theorem}
\newaliascnt{lemma}{theorem}
\newaliascnt{corollary}{theorem}
\newaliascnt{observation}{theorem}
\newaliascnt{definition}{theorem}
\newaliascnt{fact}{theorem}
\newaliascnt{remark}{theorem}
\newaliascnt{example}{theorem}
\newaliascnt{property}{theorem}
\theoremstyle{plain}
\newtheorem{proposition}[proposition]{Proposition} 
\newtheorem{lemma}[lemma]{Lemma} 
\newtheorem{corollary}[corollary]{Corollary} 
\newtheorem{observation}[observation]{Observation}
\newtheorem{fact}[fact]{Fact} 
\theoremstyle{definition}
\newtheorem{definition}[definition]{Definition} 
\newtheorem{remark}[remark]{Remark}
\newtheorem*{definition*}{Definition}
\newtheorem*{remark*}{Remark}
\newtheorem*{example*}{Example}
\theoremstyle{property}

\theoremstyle{remark}
\aliascntresetthe{proposition}
\aliascntresetthe{lemma}
\aliascntresetthe{corollary}
\aliascntresetthe{observation}
\aliascntresetthe{definition}
\aliascntresetthe{fact}
\aliascntresetthe{remark}
\aliascntresetthe{example}
\aliascntresetthe{property}

\theoremstyle{plain}
\newtheorem*{theorem.1}{Theorem \ref{thm:E-0-def.of.val.ring}}
\title{An existential $\emptyset$-definition of $\mathbb{F}_{q}[[t]]$ in $\mathbb{F}_{q}((t))$}
\author{Will Anscombe and Jochen Koenigsmann}
\date{}
\begin{document}
\maketitle
\let\thefootnote\relax\footnotetext{Most of this research forms part of the first author's doctoral thesis, which was supported by EPSRC and completed under the supervision of the second author.}
\begin{abstract}
We show that the valuation ring $\mathbb{F}_{q}[[t]]$ in the local field $\mathbb{F}_{q}((t))$ is existentially definable in the language of rings with no parameters. The method is to use the definition of the henselian topology following the work of Prestel-Ziegler to give an $\exists$-$\mathbb{F}_{q}$-definable bounded neighbouhood of $0$. Then we `tweak' this set by subtracting, taking roots, and applying Hensel's Lemma in order to find an $\exists$-$\mathbb{F}_{q}$-definable subset of $\mathbb{F}_{q}[[t]]$ which contains $t\mathbb{F}_{q}[[t]]$. Finally, we use the fact that $\mathbb{F}_{q}$ is defined by the formula $x^{q}-x\;\dot{=}\;0$ to extend the definition to the whole of $\mathbb{F}_{q}[[t]]$ and to rid the definition of parameters.

Several extensions of the theorem are obtained, notably an $\exists$-$\emptyset$-definition of the valuation ring of a non-trivial valuation with divisible value group.
\end{abstract}

\section{Introduction}
This paper deals with questions of definability in power series fields. Unless stated otherwise, all definitions will be in the language $\mathcal{L}_{\mathrm{ring}}$ of rings. Let $q=p^{k}$ be a power of a prime and let $\mathbb{F}_{q}((t))$ be the field of formal power series over the finite field $\mathbb{F}_{q}$; sometimes this is called the field of Laurent series over $\mathbb{F}_{q}$. The ring $\mathbb{F}_{q}[[t]]$ of formal power series is the valuation ring of the $t$-adic valuation on $\mathbb{F}_{q}((t))$. 

In \autoref{section:E-0-def} of this paper we prove the following theorem.
\begin{theorem}\label{thm:E-0-def.of.val.ring}
$\mathbb{F}_{q}[[t]]$ is existentially definable in $\mathbb{F}_{q}((t))$ using no parameters.
\end{theorem}
This result fits into a long history of definitions of valuation rings in valued fields. In the particular case of power series fields, a lot is already known.
\begin{observation}\label{prp:other.power.series.fields}
$K[[t]]$ is not $\exists$-$\emptyset$-definable in $K((t))$ for $K=\mathbb{Q}_{p},\mathbb{C}$.
\end{observation}
\begin{proof}
Let $K((t))^{\mathrm{Px}}:=\bigcup_{n<\omega}K((t^{1/n}))$ be the field of Puiseux series and let $K[[t]]^{\mathrm{Px}}:=\bigcup_{n<\omega}K[[t^{1/n}]]$. If $K[[t]]$ is $\exists$-$\emptyset$-definable in $K((t))$ then $K[[t]]^{\mathrm{Px}}$ is $\exists$-$\emptyset$-definable in $K((t))^{\mathrm{Px}}$ by the same formula. If $K=\mathbb{C}$ then, by Puiseux's Theorem, $\mathbb{C}((t))^{\mathrm{Px}}$ is algebraically closed and thus no infinite co-infinite subsets are definable. In particular, $\mathbb{C}[[t]]^{\mathrm{Px}}$ is not definable.

Now let $K=\mathbb{Q}_{p}$ and let $\phi$ be an existential formula (with no parameters). Suppose that $\phi$ defines $\mathbb{Q}_{p}[[t]]$ in $\mathbb{Q}_{p}((t))$; then in $\mathbb{Q}_{p}((t))^{\mathrm{Px}}$ the formula $\phi$ defines $\mathbb{Q}_{p}[[t]]^{\mathrm{Px}}$, which is a proper subring. Note also that $\mathbb{Q}_{p}$ is contained in this definable set. The field $\mathbb{Q}_{p}((t))^{\mathrm{Px}}$ is $p$-adically closed, thus $\mathbb{Q}_{p}\preceq\mathbb{Q}((t))^{\mathrm{Px}}$. Thus $\phi$ defines $\mathbb{Q}_{p}$ in $\mathbb{Q}_{p}$, which is \em not \rm a proper subset. This contradicts the elementary equivalence of $\mathbb{Q}_{p}$ and $\mathbb{Q}_{p}((t))^{\mathrm{Px}}$.
\end{proof}

In the field $\mathbb{Q}_{p}$ the valuation ring $\mathbb{Z}_{p}$ is $\exists$-$\emptyset$-definable by the formula $\exists y\;1+x^{l}p=y^{l}$, for any prime $l\neq p$. This formula is not, however, uniform in $p$. Analogies between $\mathbb{Q}_{p}$ and $\mathbb{F}_{p}((t))$ naturally suggest the first, `folkloric' definition of $\mathbb{F}_{q}[[t]]$ in $\mathbb{F}_{q}((t))$, which is given in the following fact.

\begin{fact}\label{prp:folklore}
$\mathbb{F}_{q}[[t]]$ is defined in $\mathbb{F}_{q}((t))$ by the existential formula $\exists y\;1+x^{l}t=y^{l}$, for any prime $l$ such that $l\nmid q$.
\end{fact}
\begin{proof}
Let $\mathcal{O}:=\mathbb{F}_{q}[[t]]$ denote the valuation ring and $\mathcal{M}:=t\mathcal{O}$ the maximal ideal. Suppose that $x\in\mathcal{O}$. Clearly $1+x^{l}t\in 1+\mathcal{M}=(1+\mathcal{M})^{l}$ by henselianity. Conversely, suppose $x$ is such that $v_{t}x<0$. Then $v_{t}x^{l}\leq-l$ and $v_{t}(x^{l}t)\leq 1-l<0$. Thus $v_{t}(1+x^{l}t)=v_{t}(x^{l}t)=1+lv_{t}x$ cannot be divisible by $l$ and there can exist no $y$ such that $1+x^{l}t=y^{l}$.
\end{proof}
Other definitions are also well-known. One example is an $\exists\forall\exists\forall$-definition with no parameters due to Ax, from \cite{Ax65}, which applies to all power series fields.
\begin{fact}\label{prp:Ax.def}\bf (Implicit in \cite{Ax65}) \rm\em
Let $F$ be any field. Then $F[[t]]$ is $\exists\forall\exists\forall$-$\emptyset$-definable in $F((t))$.
\end{fact}
Another definition, in even greater generality, which uses no parameters is due to the second author and is from \cite{Koenigsmann04}. However, this definition is not existential.
\begin{fact}\label{prp:Koe.def}\bf (Lemma 3.6, \cite{Koenigsmann04}) \rm\em
Let $F$ be any field and suppose that $\mathcal{O}$ is an henselian rank $1$ valuation ring on $F$ with a non-divisible value group. Then $\mathcal{O}$ is $\emptyset$-definable.
\end{fact}

Recent work of Cluckers-Derakhshan-Leenknegt-Macintyre on the uniformity of definitions of valuation rings in henselian valued fields includes the following theorem in the expanded language $\mathcal{L}_{\mathrm{ring}}\cup\{P_{2}\}$, where the Macintyre predicate $P_{2}$ is interpreted as the set of squares.

\begin{fact}\label{fact:CDLM.theorem.3}\rm(\bf Theorem 3, \cite{Cluckers-Derakhshan-Leenknegt-Macintyre12}\rm)\em There is an existential formula $\phi$ in $\mathcal{L}_{\mathrm{ring}}\cup\{P_{2}\}$ which defines the valuation ring in all henselian valued fields $K$ with finite or pseudo-finite residue field of characteristic not equal to $2$.
\end{fact}

One consequence of \autoref{thm:E-0-def.of.val.ring} is in the study of definability in $\mathbb{F}_{q}((t))$: it reduces questions of existential definability in the language of valued fields (for example $\mathcal{L}_{\mathrm{ring}}$ expanded with a unary prediate for the valuation ring) to existential definability in $\mathcal{L}_{\mathrm{ring}}$ conservatively in parameters; i.e. without needing more parameters.

It is famously unknown whether or not the theory of $\mathbb{F}_{p}((t))$ is decidable, whereas $\mathbb{Q}_{p}$ is decidable by the work of Ax-Kochen and Ershov. In \cite{Denef-Schoutens03} Denef and Schoutens prove that Hilbert's 10th problem has a positive solution in $\mathbb{F}_{q}[[t]]$ (in the language $\mathcal{L}_{\mathrm{ring}}\cup\{t\}$ of discrete valuation rings) on the assumption of Resolution of Singularities in characteristic $p$. As a consequence of \autoref{thm:E-0-def.of.val.ring}, we prove in \autoref{cor:H10} that Hilbert's 10th problem in $\mathcal{L}_{\mathrm{ring}}$ has a solution over $\mathbb{F}_{q}((t))$ if and only if it has a solution in $\mathbb{F}_{q}[[t]]$. Of course, the analagous result for the language $\mathcal{L}_{\mathrm{ring}}\cup\{t\}$ follows from the `folkloric' definition in \autoref{prp:folklore}.

As an imperfect field, $\mathbb{F}_{p}((t))$ cannot be model complete in the language of rings; however it is still unknown whether it is model complete in a relatively `nice' expansion of that language, for example some analogy of the Macintyre language (see \cite{Macintyre76}) suitable for positive characteristic.

\section{The $\exists$-$\emptyset$-definition of $\mathbb{F}_{q}[[t]]$ in $\mathbb{F}_{q}((t))$}
\label{section:E-0-def}

Let $v_{t}$ be the $t$-adic valuation on $\mathbb{F}_{q}((t))$. The valuation ring of $v_{t}$ is the ring $\mathbb{F}_{q}[[t]]$ of formal power series which has a unique maximal ideal $t\mathbb{F}_{q}[[t]]$. The value group of $v_{t}$ is $\mathbb{Z}$ and the reside field is $\mathbb{F}_{q}$. Importantly, the valued field $(\mathbb{F}_{q}((t)),\mathbb{F}_{q}[[t]])$ is henselian.

\subsection{Spheres and balls in valued fields}

We briefly make a few definitions and notational conventions. Let $(K,\mathcal{O})$ be a valued field, let $v$ be the corresponding valuation, and let $vK$ denote the value group.

\begin{definition}\label{def:open.and.closed.balls}
For $n\in vK$ and $a\in K$, we let
\begin{enumerate}
\item $S(n):=v^{-1}(\{n\})$ be the set of elements of value $n$,
\item $B(n;a):=a+v^{-1}((n,\infty))$ be the \em open ball \rm of radius $n$ around $a$, and
\item $\bar{B}(n;a):=a+v^{-1}([n,\infty))$ be the \em closed ball \rm of radius $n$ around $a$.
\end{enumerate}
\end{definition}
\noindent We let $\sqcup$ denote a disjoint union.
\begin{lemma}\label{lem:spheres.and.balls}
Let $n\in vK$. Then
\begin{enumerate}
\item $B(n;0)\subseteq S(n)-S(n)$,
\item $\bar{B}(n;0)= S(n)\sqcup B(n;0)$, and
\item $\bar{B}(n;0)-\bar{B}(n;0)=\bar{B}(n;0)$.
\end{enumerate}
\end{lemma}
\begin{proof}
\begin{enumerate}
\item Let $x\in B(n;0)$ and let $y\in S(n)$. Then $v(y)=n<v(x)$, so that $v(x-y)=n$ (by an elementary consequence of the ultrametric inequality) and $x-y\in S(n)$. Thus $x=x-y+y\in S(n)-S(n)$.
\item Let $x\in\bar{B}(n;0)$. Then either $v(x)=n$ or $v(x)>n$.
\item Let $x,y\in\bar{B}(n;0)$. By the ultrametric inequality $v(x-y)\geq n$. Thus $x-y\in\bar{B}(n;0)$.
\end{enumerate}
\end{proof}

\subsection{An $\exists$-definable filter base for the neighbourhood filter of zero}

Following Prestel and Ziegler in \cite{Prestel-Ziegler78}, we give the definition of a t-henselian field. From another paper of Prestel (\cite{Prestel91}), we recall a definition of the t-henselian topology (in the context of t-henselian non-separably closed fields). We obtain an $\exists$-definable bounded neighbourhood of zero. For more information on t-henselian fields, see \cite{Prestel-Ziegler78}.

For $n\in\mathbb{N}$ and any subset $U\subseteq K$, we denote $x^{n+1}+x^{n}+U[x]^{n-1}:=\{x^{n+1}+x^{n}+u_{n-1}x^{n-1}+...+x_{0}\;|\;u_{i}\in U\}$.

\begin{definition}\label{def:t-hens}
Let $K$ be any field. We say that $K$ is \em t-henselian \rm if there is a field topology $\mathcal{T}$ on $K$ induced by an absolute value or a valuation with the property that, for each $n\in\mathbb{N}$, there exists $U\in\mathcal{T}$ such that $0\in U$ and such that each $f\in x^{n+1}+x^{n}+U[x]^{n-1}$ has a root in $K$.
\end{definition}

The following definition of the t-henselian topology from \cite{Prestel91} corrects an earlier definition given in \cite{Prestel-Ziegler78}. To define a group topology, we mean that a filter base of the filter of neighbourhoods of zero is a definable family.

Let $D:=D_{x}$ denote the formal derivative with respect to the variable $x$.

\begin{lemma}\label{lem:def.of.nhbhds}\bf (Proof of Lemma, \cite{Prestel91}) \rm\em
Suppose that $K$ is t-henselian and not separably closed. Let $f\in K[x]$ be a separable irreducible polynomial without a zero in $K$. Let $a\in K\setminus Z(Df)$ be any element which is not a zero of the formal derivative of $f$. Let $U_{f,a}:={f(K)}^{-1}-{f(a)}^{-1}=\{\frac{1}{f(x)}-\frac{1}{f(a)}\;|\;x\in K\}$. Then $\mathcal{U}:=\{c\cdot U_{f,a}|c\in K^{\times}\}$ is a base for the filter of open neighbourhoods around zero in the (unique) t-henselian topology.
\end{lemma}

We prove a simple consequence of the Lemma.

\begin{proposition}\label{prp:def.of.nhbhds}
Suppose that $C\subseteq K$ is a relatively algebraically closed subfield of $K$ which is not separably closed. There exists $V\subseteq K$ which is an $\exists$-$C$-definable bounded neighbourhood of $0$ in the t-henselian topology.
\end{proposition}
\begin{proof}
We choose $f\in C[x]$ to be non-linear, irreducible, and separable. Let $n:=\mathrm{deg}(f)$; thus $\mathrm{deg}(Df)\leq n-1$. If $|C|>n-1$ then we may choose $a\in C\setminus Z(Df)$. On the other hand, if $C$ is a finite field, then $C$ allows separable extensions of degree $2$. So we may choose $f$ to be of degree $2$; whence $Df$ is of degree $\leq 1$ and again there exists $a\in C$ which is not a root of $Df$. Let $V:=U_{f,a}=f(K)^{-1}-f(a)^{-1}$. Clearly $V$ is $\exists$-$C$-definable. As discussed in \autoref{lem:def.of.nhbhds}, $V$ is a bounded neighbourhood of $0$.
\end{proof}

\subsection{An $\exists$-$F$-definable set between $\mathcal{O}$ and $\mathcal{M}$ in $F((t))$}

Now let $K:=F((t))$ be the field of formal power series over a field $F$. Let $v$ be the $t$-adic valuation, let $\mathcal{O}:=F[[t]]$ be the valuation ring of $v$, let $\mathcal{M}:=t\mathcal{O}$ be its maximal ideal, and let $vK=\mathbb{Z}$ be its value group. Note that $(K,\mathcal{O})$ is henselian. Let $C\subseteq K$ be any subset. Let $\mathcal{P}:=S(1)$ be the set of elements of value $1$; thus $\mathcal{P}$ is the set of uniformisers.

In the following proposition we show how to `tweak' a definable bounded neighbourhood of $0$ until we obtain a subset of $\mathcal{O}$ containing $\mathcal{M}$, in such a way as to preserve definability.

\begin{proposition}\label{prp:fiddling.with.V}
Suppose that $V\subseteq K$ is an $\exists$-$C$-definable bounded neighbourhood of $0$.
\begin{enumerate}
\item There exists $W\subseteq K$ which is bounded, $\exists$-$C$-definable, and is such that $\mathcal{P}\subseteq W$.
\item There exists $X\subseteq K$ which is bounded, $\exists$-$C$ definable, and is such that $\mathcal{M}\subseteq X$.
\item There exists $Y\subseteq K$ which is bounded by $\mathcal{O}$, $\exists$-$C$-definable, and is such that $\mathcal{M}\subseteq Y$.
\end{enumerate}
\end{proposition}
\begin{proof}
\begin{enumerate}
\item $V$ is a neighbourhood of $0$. Let $n\in\mathbb{Z}$ be such that $B(n;0)\subseteq V$. Without loss of generality, we suppose that $n\geq 0$. Choose any $m>n$; then $\mathcal{P}^{m}\subseteq S(m)\subseteq B(n;0)\subseteq V$. Let $\phi(x)$ be the formula expressing $x^{m}\in V$, and let $W:=\phi(K)$ be the set defined by $\phi$ in $K$. Note that $W$ is $\exists$-$C$-definable, and $\mathcal{P}\subseteq W$.

It remains to show that $W$ is bounded. Since $V$ is bounded, there exists $l\in\mathbb{Z}$ such that $V\subseteq B(l;0)$. Let $l':=\mathrm{min}\{l,-1\}$ and let $b\notin B(l';0)$. Since $vb\leq l'\leq-1<0$, we have that $vb^{m}=m vb\leq vb\leq l'\leq l$. Thus $b^{m}\notin V$ and $$\left(x^{m}\in V\implies x\in B(l';0)\right).$$ So $W\subseteq B(l';0)$.

\item Let $W':=W\cup\{0\}$ and set $X:=W-W'$. Clearly $X$ is bounded and $\exists$-$C$-definable. By \autoref{lem:spheres.and.balls}, we see that $B(1;0)\subseteq S(1)-S(1)=\mathcal{P}-\mathcal{P}\subseteq W-W\subseteq X$. Also $\mathcal{P}\subseteq W-0\subseteq X$. Thus $\mathcal{M}=\bar{B}(1;0)=\mathcal{P}\sqcup B(1;0)\subseteq X$.

\item $X$ is bounded but contains $\mathcal{M}$, so there exists $h\in\mathbb{N}$ such that $X\subseteq B(-h;0)$. Let $\psi(x)$ be the formula expressing $x^{h}\in X$, and set $Y:=\psi(K)-\psi(K)$. Observe that $Y$ is $\exists$-$C$-definable. It remains to show that $Y$ is bounded by $\mathcal{O}$ and that $\mathcal{M}\subseteq Y$.

If $va\leq-1$ then $va^{h}=hva\leq-h$. Thus if $va\leq-1$, then $a^{h}\notin B(-1,0)\supseteq X$ and $a\notin\psi(K)$. Therefore $\psi(K)\subseteq\mathcal{O}$. By \autoref{lem:spheres.and.balls}, $Y=\psi(K)-\psi(K)\subseteq\mathcal{O}-\mathcal{O}=\mathcal{O}$.

Since $\mathcal{P}^{h}\subseteq S(h)$ (where $\mathcal{P}^{h}$ is the set of $h$-th powers of elements of $\mathcal{P}$) and $S(h)\subseteq\mathcal{M}\subseteq X$; we have that $\mathcal{P}\subseteq\psi(K)$. Thus $\mathcal{P}-\mathcal{P}\subseteq\psi(K)-\psi(K)$. By \autoref{lem:spheres.and.balls}, $B(1;0)\subseteq\mathcal{P}-\mathcal{P}$; thus $B(1;0)\subseteq\psi(K)-\psi(K)$. Since $0^{h}=0\in\mathcal{M}\subseteq X$, $0\in\psi(K)$ and $\mathcal{P}-0\subseteq\psi(K)-\psi(K)$. By another application of \autoref{lem:spheres.and.balls}, this means that $\mathcal{M}=\mathcal{P}\sqcup B(1;0)\subseteq\psi(K)-\psi(K)=Y$, as required.
\end{enumerate}
\end{proof}

\subsection{The $\exists$-$\emptyset$-definition of $\mathbb{F}_{q}[[t]]$ in $\mathbb{F}_{q}((t))$}

Finally, we consider the special case where $F$ is the finite field $\mathbb{F}_{q}$ for $q$ a prime power. Thus we fix $K:=\mathbb{F}_{q}((t))$ and $\mathcal{O}:=\mathbb{F}_{q}[[t]]$.

\begin{lemma}\label{prp:bounded.nhbhd.1}
There exists an $\exists$-$\mathbb{F}_{q}$-definable bounded neighbourhood of $0$.
\end{lemma}
\begin{proof}
$\mathbb{F}_{q}\subseteq K$ is relatively algebraically closed in $K$ and is not separably closed. By \autoref{prp:def.of.nhbhds} there exists $V$ with the required properties.
\end{proof}

\begin{proposition}\label{prp:E-q-def.of.val.ring}
$\mathcal{O}$ is $\exists$-$\mathbb{F}_{q}$-definable in $K$.
\end{proposition}
\begin{proof}
We combine \autoref{prp:bounded.nhbhd.1} and \autoref{prp:fiddling.with.V} to obtain an $\exists$-$\mathbb{F}_{q}$-definable set $Y$ which contains $\mathcal{M}$ and is bounded by $\mathcal{O}$. Note that $\mathbb{F}_{q}$ is an algebraic set defined by the formula $x^{q}-x\dot{=}0$ in $K$. Let $\chi(x):=\exists y(y^{q}-y\dot{=}0\wedge x\in y+Y)$. This is obviously an $\exists$-$\mathbb{F}_{q}$-formula. Since $\mathcal{O}=\mathbb{F}_{q}+\mathcal{M}$ and $\mathcal{M}\subseteq Y\subseteq\mathcal{O}$, it is clear that $\chi(K)=\mathcal{O}$.
\end{proof}

We will improve \autoref{prp:E-q-def.of.val.ring} by removing the parameters. In the definition of the set $U_{f,a}$ we used $a$ and the coefficients of $f$ as parameters. All of these come from $\mathbb{F}_{q}$, but not necessarily from $\mathbb{F}_{p}$. Although elements of $\mathbb{F}_{q}$ are not closed terms, they are algebraic over $\mathbb{F}_{p}$. We use this algebraicity and a few simple tricks to find an existential formula with no parameters which defines $\mathcal{O}$.
\begin{fact}\label{fact:next.prime}
We state a simple consequence of Euclid's famous argument about the infinitude of the primes. Let $\{p_{i}|i\in I\}$ be a finite set of primes. There exists another prime $p'\leq\prod_{i\in I}p_{i}+1$ which is not in the set $\{p_{i}|i\in I\}$.

Now let $k\in\mathbb{N}$ and let $P$ be the set of primes that divide $k$. Of course $\prod_{p\in P}p\leq k$. By the previous remark, there exists another prime $p'\notin P$ such that $p'\leq\prod_{p\in P}p+1$. If $p'>k$ then $k=\prod_{p\in P}p$ and $p'=k+1$. Thus $p'\leq k+1$. Thus the least prime $p'$ not dividing a natural number $k$ is no greater than $k+1$. Of course $k+1$ is a very bad upper bound for $p'$ in general; although if $k=1,2$ then $p'=k+1$.
\end{fact}
\begin{lemma}\label{prp:bounded.nhbhd.2}
There exists an $\exists$-$\emptyset$-definable bounded neighbourhood of $0$.
\end{lemma}
\begin{proof}
We seek a polynomial $f\in\mathbb{F}_{p}[x]$ which is irreducible in $\mathbb{F}_{q}[x]$ and is such that not all elements of $\mathbb{F}_{q}$ are roots of $Df$, i.e. $x^{q}-x\nmid Df$.

Write $q=p^{k}$ and let $l$ be the least prime not dividing $k$. By \autoref{fact:next.prime} $l\leq k+1$; consequently $l\leq p^{k}=q$. Let $f\in \mathbb{F}_{p}[x]$ be an irreducible polynomial of degree $l$. Since $l\nmid k$, $f$ is still irreducible in $\mathbb{F}_{q}[x]$. Furthermore, $Df$ is of degree $\leq l-1<q$. Thus it cannot be the case that every element of $\mathbb{F}_{q}$ is a zero of $Df$. For any $a\in\mathbb{F}_{q}$ which is not a zero of $Df$, $U_{f,a}=f(K)^{-1}-f(a)^{-1}$ is an $\exists$-$\mathbb{F}_{q}$-definable bounded neighbourhood of $0$. We note that the only parameter in this definition not from $\mathbb{F}_{p}$ is $a$.

The union of finitely many bounded neighbourhoods of $0$ is also a bounded neighbourhood of $0$. Thus $$\zeta(y):=\exists y\;(y^{q}-y\dot{=}0\wedge\neg Df(y)\dot{=}0\wedge x\in U_{f,y})$$ is an $\exists$-$\mathbb{F}_{p}$-formula which defines the union $$V:=\bigcup\{U_{f,a}\;|\;a\in\mathbb{F}_{q},Df(a)\neq0\}.$$

Finally note that each element of $\mathbb{F}_{p}$ is the image of a closed term; thus each remaining parameter can be replaced by a closed term and we are left with an $\exists$-$\emptyset$-definition of $V$.
\end{proof}
\begin{remark}
Here is an alternative method to find an irreducible separable polynomial $f\in\mathbb{F}_{p}[x]$ and an element $a\in\mathbb{F}_{p}$ which is not a root of $Df$.\\

Let $l$ be a prime such that $p\nmid l\nmid k$. Let $g\in\mathbb{F}_{p}[x]$ be any irreducible polynomial of degree $l$. Since $l\nmid k$, $g$ is still irreducible over $\mathbb{F}_{q}$. Let $\alpha$ be a root of $g$ in a field extension. Either the coefficient of $x^{l-1}$ in $g$ is zero; or else we consider $h:=g(x-1)$, which is the minimal polynomial of $\alpha+1$. The coefficient of $x^{l-1}$ in $h$ is then $l\neq0$. Thus we may assume that the $x^{l-1}$ term in $g$ is non-zero. The polynomial $f:=x^{l}g(1/x)$ is the minimal polynomial of $1/\alpha$ and has non-zero linear term. Therefore $Df(0)\neq0$. Thus $U_{f,0}$ is an $\exists$-$\mathbb{F}_{p}$-definable bounded neighbourhood of $0$. As before, elements of $\mathbb{F}_{p}$ are closed terms, so we may remove all parameters from the definition.
\end{remark}
Finally, we prove \autoref{thm:E-0-def.of.val.ring}.
\begin{theorem.1}
$\mathcal{O}$ is $\exists$-$\emptyset$-definable in $K$.
\end{theorem.1}
\begin{proof}
From \autoref{prp:bounded.nhbhd.2} we obtain an $\exists$-$\emptyset$-definable bounded neighbourhood of $0$. Using again \autoref{prp:fiddling.with.V}, we obtain an $\exists$-$\emptyset$-definable set $Y$ which contains $\mathcal{M}$ and is bounded by $\mathcal{O}$. We define $\chi$ as before:$$\chi(x):=\exists y\;(y^{q}-y\dot{=}0\wedge x\in y+Y).$$This is an $\exists$-formula with no parameters and it defines $\mathcal{O}$.
\end{proof}
Nevertheless the formula still depends on $\mathbb{F}_{q}$ in several ways: our choices of $m$ and $h$ in \autoref{prp:fiddling.with.V} and our choice of $f$ in \autoref{thm:E-0-def.of.val.ring} depend on $\mathbb{F}_{q}$. The number $q$ also appears directly in several of the formulas. All these factors tell us that $\chi$ is highly non-uniform in $q$. In fact, in recent as-yet-unpublished joint work of Cluckers, Derakhshan, Leenknegt, and Macintyre (\cite{Cluckers-Derakhshan-Leenknegt-Macintyre12}) it is shown that no definition exists which is uniform in $p$ or in $k$ (where $q=p^{k}$).
\begin{remark}
With a little more effort we can be more explicit about the formula $\chi$. Suppose for the moment that $K=\mathbb{F}_{p}((t))$. Let $\wp:=x^{p}-x$ and let $f:=\wp-1$. Observe that $\wp-1$ is separable and irreducible in $K[x]$ and $Df(1)=D(\wp)(1)=-1\neq0$. Working back through the formulas and rearranging, we find that$$\chi(x):=\exists ab(x_{i}y_{i})_{i=1}^{4}\;\left(\begin{array}{l}\wp(x-a+b)\dot{=}0\;\wedge\;a^{h}\dot{=}x_{1}-x_{2}\;\wedge\\b^{h}\dot{=}x_{3}-x_{4}\;\wedge\bigwedge_{i=1}^{4}f(y_{i})(x_{i}^{m}-1)-1\dot{=}0\end{array}\right).$$
\end{remark}

\section{Extensions of the result}

\subsection{The field $\bigcup_{n\in\mathbb{N}}\mathbb{F}_{q}((t^{1/n}))$ of Puiseux series}

Let $K^{\mathrm{Px}}:=\bigcup_{n\in\mathbb{N}}\mathbb{F}_{q}((t^{1/n}))$ denote the field of Puiseux series over $\mathbb{F}_{q}$, where $(t^{1/n})_{n\in\mathbb{N}}$ is a compatible system of $n$-th roots of $t$ (for $n\in\mathbb{N}$). Note that $K^{\mathrm{Px}}$ can be formally defined as a direct limit. Let $\mathcal{O}^{\mathrm{Px}}:=\bigcup_{n\in\mathbb{N}}\mathbb{F}_{q}[[t^{1/n}]]$ denote the valuation ring of the $t$-adic valuation. Note that the value group is $\mathbb{Q}$.

The following theorem is the first example of an $\exists$-$\emptyset$-definition of a non-trivial valuation ring with divisible value group.

\begin{theorem}\label{thm:Puiseux}
$\mathcal{O}^{\mathrm{Px}}$ is $\exists$-$\emptyset$-definable in $K^{\mathrm{Px}}$.
\end{theorem}
\begin{proof}
By \autoref{thm:E-0-def.of.val.ring}, we may let $\chi$ be an $\exists$-formula (with no parameters) which defines $\mathcal{O}$ in $K$. In each field $\mathbb{F}_{q}((t^{1/n}))$ the formula $\chi$ defines the valuation ring $\mathbb{F}_{q}[[t^{1/n}]]$ since each of these fields is isomorphic to $\mathbb{F}_{q}((t))$. In the union, $\chi$ defines the union of the valuation rings (in any union of structures an existential formula defines the unions of sets that it defines in each of the structures). Thus $\chi$ defines $\mathcal{O}^{\mathrm{Px}}=\bigcup_{n\in\mathbb{N}}\mathbb{F}_{q}[[t^{1/n}]]$, as required.
\end{proof}

\subsection{The perfect hull $\mathbb{F}_{q}((t))^{\mathrm{perf}}$}

We still denote $K:=\mathbb{F}_{q}((t))$. Let $K^{\mathrm{perf}}:=\bigcup_{n\in\mathbb{N}}\mathbb{F}_{q}((t^{p^{-n}}))$ be the \em perfect hull \rm of $K$; this is also formally defined as a direct limit. Now we use \autoref{thm:E-0-def.of.val.ring} to existentially define the valuation ring $\mathcal{O}^{\mathrm{perf}}:=\bigcup_{n<\omega}\mathcal{O}^{p^{-n}}$ in $K^{\mathrm{perf}}$.

\begin{theorem}\label{thm:perfect.hull}
$\mathcal{O}^{\mathrm{perf}}$ is $\exists$-$\emptyset$-definable in $K^{\mathrm{perf}}$.
\end{theorem}
\begin{proof}
The proof is almost identical to the proof of \autoref{thm:Puiseux}.
\end{proof}

\subsection{Consequences for $\exists$-definability in $\mathcal{L}_{\mathrm{val}}$}

We return to the field $K:=\mathbb{F}_{q}((t))$. The most important consequence of \autoref{thm:E-0-def.of.val.ring} is that questions of existential definability in $\mathcal{L}_{\mathrm{val}}$ reduce to questions of existential definability in $\mathcal{L}_{\mathrm{ring}}$. Let $C\subseteq\mathbb{F}_{q}((t))$ be any subfield of parameters and let $\mathcal{L}_{\mathrm{val}}:=\mathcal{L}_{\mathrm{ring}}\cup\{\mathcal{O}\}$ be the language of valued fields.

\begin{proposition}\label{prp:reduction.of.the.problem}
Let $\alpha\in\mathcal{L}_{\mathrm{val}}$ be an existential formula with parameters from $C$. Then there exists $\beta\in\mathcal{L}_{\mathrm{ring}}$ with parameters in $C$ such that $\alpha$ and $\beta$ are equivalent modulo the theory of $\mathbb{F}_{q}((t))$.
\end{proposition}
\begin{proof}
Let $\mathbf{b}=(b_{i})_{i<q}$ be some indexing of the field $\mathbb{F}_{q}$ such that $b_{0}=0$. Let $\phi$ be a quantifier-free formula in free variables $\mathbf{y}=(y_{i})_{i<q}$ expressing the quantifier-free type of $\mathbf{b}$. We define $$\psi\;:=\;\exists\mathbf{y}\;\Big(x\in\mathcal{O}\wedge\phi(\mathbf{y})\wedge\bigwedge_{0<i<q}y_{i}+x\in\mathcal{O}^{-1}\Big).$$We claim that $\psi$ existentially defines $\mathcal{M}$. Let $a\in\mathcal{O}$. Then $a\in\mathcal{M}$ if and only if, for each $b\in\mathbb{F}_{q}^{\times}$, $a+b\in\mathcal{O}\setminus\mathcal{M}=\mathcal{O}^{\times}$; that is if and only if $K\models\psi(a)$. Thus $\psi$ is an $\exists$-$\emptyset$-definition for $\mathcal{M}$. Consequently, $K\setminus\mathcal{O}=(\mathcal{M}\setminus\{0\})^{-1}$ is $\exists$-$\emptyset$-definable; and so $\mathcal{O}$ is $\forall$-$\emptyset$-definable.

Since $\mathcal{O}$ is both $\forall$-$\emptyset$-definable and $\exists$-$\emptyset$-definable, we may convert any $\exists$-$C$-formula $\alpha$ of $\mathcal{L}_{\mathrm{val}}$ into an $\exists$-$C$-formula $\beta$ of $\mathcal{L}_{\mathrm{ring}}$.
\end{proof}

\begin{corollary}\label{cor:H10}
Hilbert's 10th problem has a solution over $\mathbb{F}_{q}((t))$ if and only if it does so over $\mathbb{F}_{q}[[t]]$, in any language which expands the language of rings.
\end{corollary}
\begin{proof}
Let $\phi$ be a quantifier-free formula with $\mathbf{x}$ the tuple of free-variables. Suppose that Hilbert's 10th problem (H10) has a solution over $\mathbb{F}_{q}((t))$. In order to decide the existential sentence $\exists\mathbf{x}\;\phi(\mathbf{x})$ in $\mathbb{F}_{q}[[t]]$ we apply our algorithm for $\mathbb{F}_{q}((t))$ to the sentence $$\exists\mathbf{x}\;\bigg(\phi(\mathbf{x})\wedge\bigwedge_{x\in\mathbf{x}}\mathcal{O}(x)\bigg),$$where $\mathcal{O}$ denotes the existential formula defining $\mathbb{F}_{q}[[t]]$ in $\mathbb{F}_{q}((t))$.

Conversely, suppose that H10 has a solution over $\mathbb{F}_{q}[[t]]$. By standard equivalences in the theory of fields we may assume that $\phi=f\dot{=}0$ for some polynomial $f\in\mathbb{F}_{p}[\mathbf{x}]$.

We need to find a quantifier-free formula which is realised in $\mathbb{F}_{q}[[t]]$ if and only if $f$ has a zero in $\mathbb{F}_{p}((t))$. For a variable $x\in\mathbf{x}$ we let $d_{x}$ denote the degree of $f$ in $x$; and for any subtuple $\mathbf{x}'\subseteq\mathbf{x}$ we let $\mathbf{x}'':=(\mathbf{x}\setminus\mathbf{x}')\cup\{x^{-1}|x\in\mathbf{x}'\}$ be a new tuple formed from $\mathbf{x}$ by inverting the elements of $\mathbf{x}'$. Then we set $f_{\mathbf{x}'}:=f(\mathbf{x}'')\prod_{x\in\mathbf{x}'}x^{d_{x}}$. Importantly, $f_{\mathbf{x}'}$ is a polynomial. Finally we let $$\phi':=\bigvee_{\mathbf{x}'\subseteq\mathbf{x}}\bigg(f_{\mathbf{x}'}\dot{=}0\wedge\bigwedge_{x\in\mathbf{x}'}\neg x\dot{=}0\bigg).$$Then $\mathbb{F}_{q}((t))\models\exists\mathbf{x}\;f(\mathbf{x})\dot{=}0$ if and only if $\mathbb{F}_{q}[[t]]\models\exists\mathbf{x}\;\phi'(\mathbf{x})$. Therefore, in order to decide $\exists\mathbf{x}\;\phi(\mathbf{x})$ in $\mathbb{F}_{q}((t))$ we apply our algorithm for $\mathbb{F}_{q}[[t]]$ to the existential sentence $\exists\mathbf{x}\;\phi'(\mathbf{x})$.
\end{proof}
A simple consequence of the `folkloric' definition of $\mathbb{F}_{q}[[t]]$ from \autoref{prp:folklore} is that \autoref{cor:H10} holds for any language expanding $\mathcal{L}_{\mathrm{ring}}\cup\{t\}$.

Note that, by the theorem of Denef-Schoutens in \cite{Denef-Schoutens03}, Hilbert's 10th problem has a positive solution in $\mathbb{F}_{p}[[t]]$ in the language $\mathcal{L}_{\mathrm{ring}}\cup\{t\}$ on the assumption of Resolution of Singularities in positive characteristic.

If Hilbert's 10th problem could be proved - outright - to have a positive solution in $\mathbb{F}_{p}[[t]]$ simply in the language of rings, then \autoref{cor:H10} would `lift' that result to $\mathbb{F}_{p}((t))$.

\def\bibfont{\footnotesize}

\null
\noindent
Will Anscombe\\
Lincoln College, Turl Street, Oxford OX1 3DR, UK\\
\tt anscombe@maths.ox.ac.uk\rm
\linebreak
\linebreak
Jochen Koenigsmann\\
Mathematical Institute, 24-29 St Giles’, Oxford OX1 3LB, UK\\
\tt koenigsmann@maths.ox.ac.uk\rm
\end{document}